\documentclass[12pt, reqno]{amsart}
\usepackage{amsmath, amsthm, amscd, amsfonts, amssymb, graphicx, color}
\usepackage[bookmarksnumbered, colorlinks, plainpages]{hyperref}
\hypersetup{colorlinks=true,linkcolor=red, anchorcolor=green, citecolor=cyan, urlcolor=red, filecolor=magenta, pdftoolbar=true}

\newtheorem{theorem}{Theorem}[section]
\newtheorem{lemma}[theorem]{Lemma}

\newtheorem{corollary}[theorem]{Corollary}
\theoremstyle{definition}
\newtheorem{definition}[theorem]{Definition}

\theoremstyle{remark}
\newtheorem{remark}[theorem]{Remark}
\numberwithin{equation}{section}

\begin{document}

\setcounter{page}{1}

\title[The polar decomposition and $n$-centered operators]{The polar decomposition for adjointable operators on Hilbert $C^*$-modules and $n$-centered operators}

\author[N. Liu, W. Luo, Q. Xu]{Na Liu, Wei Luo, \MakeLowercase{and} Qingxiang Xu$^{*}$}

\address{$^{}$Department of Mathematics, Shanghai Normal University, Shanghai 200234, PR China.}
\email{\textcolor[rgb]{0.00,0.00,0.84}{liunana0616@163.com;
luoweipig1@163.com; qingxiang\_xu@126.com}}







\begin{abstract}
Let $n$ be any natural number. The $n$-centered operator is introduced for adjointable operators on Hilbert $C^*$-modules.
Based on the characterizations of  the polar decomposition for the product of two adjointable operators, $n$-centered operators, centered operators as well as binormal operators are clarified, and some results known for the Hilbert space operators are improved.  It is proved that for an adjointable operator $T$, if $T$ is Moore-Penrose invertible and is $n$-centered, then its Moore-Penrose inverse is also $n$-centered. A Hilbert space operator $T$ is constructed such that $T$ is $n$-centered, whereas it fails to be $(n+1)$-centered.\\
\end{abstract} \maketitle

\section{\textbf{Introduction}}

 Much progress has been made in the study of the polar decomposition and its applications both for Hilbert space operators
 \cite{Furuta,Gesztesy-Malamud-Mitrea-Naboko,Halmos,Ichinose-Iwashita,Stochel-Szafraniec}, and for adjointable operators on Hilbert $C^*$-modules \cite{Frank-Sharifi,Gebhardt-Schm¨¹dgen,Karizaki-Hassani-Amyari,Szafraniec,Wegge-Olsen}. Let $H$ be a Hilbert space and $\mathbb{B}(H)$ be the set of bounded linear operators on $H$. For any $T\in\mathbb{B}(H)$, let $T=U|T|$ be the polar decomposition of $T$. One application of the polar decomposition is the study of the centered operator, which was introduced by Morrel and Muhly  in  \cite{Morrel}.  When $T$ is a centered operator, some characterizations of $|T^n|$ for  $n\in\mathbb{N}$ were carried out in \cite{Morrel} and through which it was proved therein that
 \begin{equation}\label{equ:polar decomposition all n}T^n=U^n|T^n|\ \mbox{is the polar decomposition for all $n\in\mathbb{N}$}.\end{equation}
By clarifying the polar decomposition for the product of two operators, the authors in \cite{Ito-Yamazaki-Yanagida} showed that \eqref{equ:polar decomposition all n} is  also  sufficient for a Hilbert space operator to be centered. Among other things, many other equivalent conditions for the centered operators were established in \cite{Ito} for Hilbert space operators.
These equivalent conditions were inspected for Hilbert $C^*$-module operators in our
recent paper \cite{Liu-Luo-Xu} and they are proved to be also true for an adjointable operator whenever its polar decomposition is guaranteed.

Another application of the polar decomposition is the study of the binormal operator (a kind of operators more general than the centered operators), which was introduced by Campbell in \cite{Campbell-1} and was studied in many literatures; see \cite{Campbell-2,Ito} for example.
The purpose of this paper is, in the general setting of adjointable operators on $C^*$-modules, to give a new insight into \eqref{equ:polar decomposition all n}  and give some new characterizations of the binormal operators.

Let $n$ be any natural number. The $n$-centered operator is introduced  in this paper for adjointable operators on Hilbert $C^*$-modules.
Based on the characterizations of the polar decomposition for the product of two adjointable operators, $n$-centered operators, centered operators as well as binormal operators are clarified. As a result, the main results of \cite{Ito-Yamazaki-Yanagida} are improved in the general setting of adjointable operators on Hilbert $C^*$-modules; see Remarks~\ref{rem:improvement-1}, \ref{rem:improvement-2} and \ref{rem:improvement-3}, and Corollary~\ref{thm:another result for centered operator}. Also, some new characterizations of the binormal operators are provided with parameters $\alpha$ and $\beta$; see
Theorem~\ref{thm:equivalent conditions of binormal} and  Remark~\ref{rem:improvement-4}.

In addition for any natural number $n$, it is proved in this paper that for an adjointable operator $T$, if $T$ is Moore-Penrose invertible and is $n$-centered, then its Moore-Penrose inverse is also $n$-centered. Furthermore, a Hilbert space operator $T$ is constructed such that $T$ is $n$-centered, whereas it fails to be $(n+1)$-centered.

The paper is organized as follows. In Section~\ref{sec:some preliminaries}, we recall some basic knowledge about Hilbert $C^*$-modules, and provide some elementary results about the commutativity, the range closures as well as the polar decomposition for adjointable operators. In Section~\ref{sec:the polar decomposition of product}, we study the polar decomposition for the product of two adjointable operators. The $n$-centered operators, centered operators and binormal operators on Hilbert $C^*$-modules are studied in Section~\ref{sec:binormal-Aluthge-n centered}. In Section~\ref{sec:M-P inverse of $n$-centered operator}, we study the Moore-Penrose inverse of $n$-centered operators. In the last section, we focus on the construction of a Hilbert space operator $T$ such that $T$ is $n$-centered, whereas it is not $(n+1)$-centered.

\section{\textbf{Some preliminaries}}\label{sec:some preliminaries}

Hilbert $C^*$-modules are generalizations of Hilbert spaces by allowing inner products to take values in some $C^{*}$-algebras instead of the complex  field. Let $\mathfrak{A}$ be a $C^*$-algebra. An
inner-product $\mathfrak{A}$-module is a linear space $E$ which is a right
$\mathfrak{A}$-module, together with a map $(x,y)\to \big<x,y\big>:E\times E\to
\mathfrak{A}$ such that for any $x,y,z\in E, \alpha, \beta\in \mathbb{C}$ and
$a\in \mathfrak{A}$, the following conditions hold:
\begin{enumerate}
\item[(i)] $\langle x,\alpha y+\beta
z\rangle=\alpha\langle x,y\rangle+\beta\langle x,z\rangle;$

\item[(ii)] $\langle x, ya\rangle=\langle x,y\rangle a;$

\item[(iii)] $\langle y,x\rangle=\langle x,y\rangle^*;$

\item[(iv)] $\langle x,x\rangle\ge 0, \ \mbox{and}\ \langle x,x\rangle=0\Longleftrightarrow x=0.$
\end{enumerate}

An inner-product $\mathfrak{A}$-module $E$ which is complete with respect to
the induced norm ($\Vert x\Vert=\sqrt{\Vert \langle x,x\rangle\Vert}$ for
$x\in E$) is called a (right) Hilbert $\mathfrak{A}$-module.

Suppose that $H$ and $K$ are two Hilbert $\mathfrak{A}$-modules, let ${\mathcal L}(H,K)$
be the set of operators $T:H\to K$ for which there is an operator $T^*:K\to
H$ such that $$\langle Tx,y\rangle=\langle x,T^*y\rangle \ \mbox{for any
$x\in H$ and $y\in K$}.$$  We call ${\mathcal
L}(H,K)$ the set of adjointable operators from $H$ to $K$. For any
$T\in {\mathcal L}(H,K)$, the range and the null space of $T$ are denoted by
${\mathcal R}(T)$ and ${\mathcal N}(T)$, respectively. In case
$H=K$, ${\mathcal L}(H,H)$ which is abbreviated to ${\mathcal L}(H)$, is a
$C^*$-algebra. Let ${\mathcal L}(H)_{sa}$ and ${\mathcal L}(H)_+$ be the set of  self-adjoint elements and positive elements
in $\mathcal{L}(H)$, respectively. The notation $T\ge 0$ is also used to indicate that $T$ is an element of ${\mathcal L}(H)_+$.

\begin{definition}\label{def:commutator of a and b}
 Let  $\mathfrak{B}$  be a $C^*$-algebra. The set of positive elements of $\mathfrak{B}$ is denoted by $\mathfrak{B}_+$.  For any $a,b\in \mathfrak{B}$,  let $[a,b]=ab-ba$ be the commutator of $a$ and $b$.
\end{definition}
\begin{lemma} \label{lem:commutative property extended-1} Let $a$ and $b$ be two elements in a $C^*$-algebra $\mathfrak{B}$.
\begin{enumerate}
\item[{\rm(i)}] If $a=a^*$ and $[a,b]=0$, then $[f(a),b]=0$  whenever $f$ is a continuous complex-valued  function on the interval $[-\Vert a\Vert, \Vert a\Vert]$;
\item[{\rm(ii)}] If $a,b\in \mathfrak{B}_+$, then $ab\in \mathfrak{B}_+$ if and only if \,$[a,b]=0$.
\end{enumerate}
\end{lemma}
\begin{proof} (i) A proof can be found in  \cite[Proposition~2.3]{Liu-Luo-Xu}.

(ii) If  $ab\in \mathfrak{B}_+$, then $ab=(ab)^*=ba$, hence $[a,b]=0$. Conversely, if $[a,b]=0$, then
by (i)  we have $ab=a^\frac12(a^\frac12 b)=a^\frac12 b\, a^\frac12 \in \mathfrak{B}_+$.
\end{proof}

Throughout the rest of this paper, $\mathfrak{A}$ is a $C^*$-algebra, $E,H$ and $K$ are Hilbert $\mathfrak{A}$-modules.

\begin{definition}
A closed
submodule $M$ of  $H$ is said to be
orthogonally complemented  if $H=M\dotplus M^\bot$, where
$$M^\bot=\big\{x\in H: \langle x,y\rangle=0\ \mbox{for any}\ y\in
M\big\}.$$
In this case, the projection from $H$ onto $M$ is denoted by $P_M$.
\end{definition}

An elementary result on the commutativity of adjointable operators is as follows:

\begin{lemma}\label{lem:commutative property extended-3}{\rm \cite[Propositions~2.4 and 2.6]{Liu-Luo-Xu}} Let $S\in \mathcal{L}(H)$ and $T\in\mathcal{L}(H)_+$ be such that $[S,T]=0$. Then the following statements are valid:
\begin{enumerate}
\item[{\rm (i)}] $[S,T^\alpha]=0$ for any $\alpha>0$;
\item[{\rm (ii)}] If in addition $\overline{\mathcal{R}(T)}$ is orthogonally complemented, then $\big[S,  P_{\overline{\mathcal{R}(T)}}\big]=0$.
\end{enumerate}
\end{lemma}

A direct application of Lemma~\ref{lem:commutative property extended-3} (ii) is as follows:
\begin{lemma}\label{prop:commutative property extended-4} Let $S,T\in\mathcal{L}(H)_+$ be such that both
$\overline{\mathcal{R}(S)}$ and $\overline{\mathcal{R}(T)}$ are orthogonally complemented. If $[S,T]=0$, then
\begin{equation*}\big[S, P_{\overline{\mathcal{R}(T)}}\big]=\big[T, P_{\overline{\mathcal{R}(S)}}\big]=\big[P_{\overline{\mathcal{R}(S)}},P_{\overline{\mathcal{R}(T)}}\big]=0.
\end{equation*}
\end{lemma}

Next, we state two results about the range closures of adjointable operators.

\begin{lemma}\label{lem:rang characterization-1} {\rm \cite[Proposition 2.7]{Liu-Luo-Xu}} Let $A\in\mathcal{L}(H,K)$ and $B,C\in\mathcal{L}(E,H)$ be such that $\overline{\mathcal{R}(B)}=\overline{\mathcal{R}(C)}$. Then $\overline{\mathcal{R}(AB)}=\overline{\mathcal{R}(AC)}$.
\end{lemma}

\begin{lemma}\label{lem:Range Closure of Ta and T} {\rm (\cite[Proposition 3.7]{Lance} and \cite[Proposition 2.9]{Liu-Luo-Xu})} Let $T\in\mathcal{L}(H,K)$.
\begin{enumerate}
\item[{\rm(i)}] Then $\overline {\mathcal{R}(T^*T)}=\overline{ \mathcal{R}(T^*)}$ and $\overline {\mathcal{R}(TT^*)}=\overline{ \mathcal{R}(T)}$;
\item[{\rm(ii)}] If $T\in \mathcal{L}(H)_+$, then $\overline{\mathcal{R}(T^{\alpha})}=\overline{\mathcal{R}(T)}$ for any $\alpha>0$.
\end{enumerate}
\end{lemma}

\begin{definition}  An element $U$ of $\mathcal{L}(H,K)$ is said to be a partial isometry if $U^*U$ is a projection in
$\mathcal{L}(H)$.
\end{definition}

\begin{lemma}\label{prop:basic property of partial isometry} {\rm \cite[Proposition~3.2]{Liu-Luo-Xu}}\ Let $U\in\mathcal{L}(H,K)$  be a partial isometry. Then $U^*$ is also a partial isometry such that $UU^*U=U$.
\end{lemma}

For any $T\in\mathcal{L}(H,K)$, let $|T|$ denote the square root of $T^*T$. That is, $|T|=(T^*T)^\frac12$ and $|T^*|=(TT^*)^\frac12$.
\begin{definition}\label{defn:defn of polar decomposition} {\rm \cite[Definition 3.10]{Liu-Luo-Xu}}
The polar decomposition of $T\in\mathcal{L}(H,K)$ can be characterized as
\begin{equation}\label{equ:two conditions of polar decomposition}T=U|T|\ \mbox{and}\ U^*U=P_{\overline{\mathcal{R}(T^*)}},\end{equation}
where $U\in\mathcal{L}(H,K)$ is a partial isometry.
\end{definition}

\begin{lemma}\label{thm:existence condition for polar decomposition} {\rm \cite[Lemma~3.6 and Theorem 3.8]{Liu-Luo-Xu}} Let $T\in\mathcal{L}(H,K)$. Then the following two statements are equivalent:
\begin{enumerate}
\item[{\rm (i)}]  $\overline{\mathcal{R}(T^*)}$ and $\overline{\mathcal{R}(T)}$ are both orthogonally complemented;
\item[{\rm (ii)}] There exists $U\in\mathcal{L}(H,K)$ such that
\begin{equation}\label{equ:defn of polar decomposition for T}T=U|T| \ \mbox{and}\  U^*U=P_{\overline{\mathcal{R}(T^*)}}.\end{equation}
\end{enumerate}
In each case, it holds that
\begin{eqnarray}\label{equ:the polar decomposition of T star-pre stage}T^*&=&U^*|T^*| \ \mbox{and}\  UU^*=P_{\overline{\mathcal{R}(T)}},\\
\label{eqn:key telationship between two 1/2}|T^*|&=&U|T| U^*\ \mbox{and}\ U|T|=|T^*|U.\end{eqnarray}
\end{lemma}

\begin{remark}\label{rem:existence of polar decomposition} It follows from Lemma~\ref{thm:existence condition for polar decomposition} and \cite[Lemma~3.9]{Liu-Luo-Xu} that $T\in\mathcal{L}(H,K)$ has the (unique) polar decomposition if and only if $\overline{\mathcal{R}(T^*)}$ and $\overline{\mathcal{R}(T)}$ are both orthogonally complemented, and in this case $T^*=U^*|T^*|$ is the polar decomposition of $T^*$.
\end{remark}

\section{\textbf{The polar decomposition for the product of two operators}}\label{sec:the polar decomposition of product}

In this section, we study the polar decomposition for the product of two adjointable operators on Hilbert $C^*$-modules.

\begin{lemma}\label{lem:polar decomposition of product of operator-1}
{\rm (cf.\,\cite[Theorem~2.1]{Ito-Yamazaki-Yanagida})}\   Let $T=U|T|$ and $S=V|S|$ be the polar decompositions of $T,S\in\mathcal{L}(H)$, respectively. Then $TS$ has the polar decomposition if and only if $|T| |S^*|$ has the polar decomposition.
\end{lemma}

\begin{proof} ``$\Longleftarrow$":\ Suppose  $|T||S^*|$ has the polar decomposition $|T||S^*|=W_1\big||T|\,|S^*|\big|$. We prove that
\begin{equation}\label{equ:the polar decomposition of TS}TS=UW_1V|TS|\end{equation}
is the polar decomposition of $TS$.

First, we show that $X=UW_1V$ is a partial isometry.  By \eqref{equ:two conditions of polar decomposition} and \eqref{equ:the polar decomposition of T star-pre stage}, we have
\begin{eqnarray*}W_1^*W_1&=&P_{\overline{\mathcal{R}(|S^*||T|)}}\le  P_{\overline{\mathcal{R}(S)}}=VV^*,\\
W_1W_1^*&=&P_{\overline{\mathcal{R}(|T||S^*|)}}\le P_{\overline{\mathcal{R}(T^*)}}=U^*U.
\end{eqnarray*}
Therefore, $U^*UW_1=W_1$ and $VV^*W_1^*W_1=W_1^*W_1$. It follows that
\begin{eqnarray*}X^*X=V^*W_1^*W_1V\ \mbox{and}\ (X^*X)^2=X^*X,
\end{eqnarray*}
hence $X^*X$ is a projection, i.e., $X$ is a partial isometry.

Next, we prove that equation~\eqref{equ:the polar decomposition of TS} is true. Note that  $V|S|=S$,  so we have
\begin{eqnarray*}\big||T||S^*|\big|^2=|S^*|\,|T|^2\,|S^*|=V|S|V^*\cdot |T|^2\cdot V|S|V^*=V\cdot |TS|^2\cdot V^*,
\end{eqnarray*}
which means by $V^*V=P_{\overline{\mathcal{R}(S^*)}}$ that
\begin{equation}\label{equ:formula through |TS|}\big||T||S^*|\big|=V\cdot |TS|\cdot V^*,\ \mbox{hence}\ V\cdot |TS|=\big||T||S^*|\big|\cdot V.
\end{equation}
It follows that
\begin{eqnarray*}UW_1V|TS|=U\cdot W_1 \big||T||S^*|\big|\cdot V=U|T|\cdot |S^*|V=TS.
\end{eqnarray*}

Finally, we prove that $X^*X=P_{\overline{\mathcal{R}(|TS|)}}$. In fact, from \eqref{equ:formula through |TS|} we can obtain
$$|TS|=V^*\cdot \big||T||S^*|\big|\cdot V=\big[V^*\big||T||S^*|\big|^\frac12\big]\cdot\big[V^*\big||T||S^*|\big|^\frac12\big]^*,$$
hence by Lemmas~\ref{lem:rang characterization-1} and \ref{lem:Range Closure of Ta and T},  we have
\begin{eqnarray*}\overline{\mathcal{R}(|TS|)}&=&\overline{\mathcal{R}(V^*\big||T||S^*|\big|^\frac12)}=\overline{\mathcal{R}(V^*\big||T||S^*|\big|)}=\overline{\mathcal{R}(V^*W_1^*W_1)}\\
&=&\overline{\mathcal{R}(V^*W_1^*W_1VV^*)}=\overline{\mathcal{R}(V^*W_1^*W_1V)}=\mathcal{R}(X^*X).
\end{eqnarray*}

``$\Longrightarrow$":\ Suppose that $TS$ has the polar decomposition $TS=W_2\big|TS\big|$. We prove that
\begin{equation}\label{equ:the polar decomposition of |T||S^*|}|T| |S^*|=U^*W_2V^*\big||T|S^*|\big|\end{equation}
is the polar decomposition of $|T||S^*|$.

First, we show that $Y=U^*W_2V^*$ is a partial isometry. Indeed,
\begin{eqnarray*}W_2^*W_2&=&P_{\overline{\mathcal{R}(S^*T^*)}}\le  P_{\overline{\mathcal{R}(S^*)}}=V^*V,\\
W_2W_2^*&=&P_{\overline{\mathcal{R}(TS)}}\le P_{\overline{\mathcal{R}(T)}}=UU^*.
\end{eqnarray*}
Therefore,  $Y^*Y$ is a projection as illustrated before.

Next, we prove that equation~\eqref{equ:the polar decomposition of |T||S^*|} is true. Taking $*$-operation, by \eqref{equ:formula through |TS|} we get
\begin{eqnarray*} |TS|V^*=V^*\big||T||S^*|\big|.
\end{eqnarray*}
It follows that
\begin{eqnarray*} U^*W_2V^*\big||T||S^*|\big|&=&U^*\cdot W_2|TS|\cdot V^*
=U^*T\cdot SV^*\\&=&U^*U|T|\cdot|S^*|VV^*=|T||S^*|.
\end{eqnarray*}

Finally, we prove that $Y^*Y=P_{\overline{\mathcal{R}\big(\big||T||S^*|\big|\big)}}$. In fact,
\begin{eqnarray*}\overline{\mathcal{R}\big(\big||T||S^*|\big|\big)}
&=&\overline{\mathcal{R}(V|TS|V^*)}=\overline{\mathcal{R}(V|TS|^{\frac{1}{2}})}
=\overline{\mathcal{R}(V|TS|)}=\overline{\mathcal{R}(VW_2^*W_2)}\\
&=&\overline{\mathcal{R}(VW_2^*W_2V^*V)}=\overline{\mathcal{R}(VW_2^*W_2V^*)}=\mathcal{R}(Y^*Y).\qedhere
\end{eqnarray*}
\end{proof}

\begin{remark}\label{rem:improvement-1}{\rm The preceding lemma was established in \cite[Theorem~2.1]{Ito-Yamazaki-Yanagida} for Hilbert space operators, where only sufficiency was considered. An interpretation of this lemma can be made by using block matrices as follows:

Suppose that $T=U|T|$ and $S=V|S|$ are the polar decompositions. Let $X, Y, Z_1, Z_2\in \mathcal{L}(H\oplus H)$ be defined respectively by
$$X=\left(
                         \begin{array}{cc}
                           T & 0 \\
                           0 & S^* \\
                         \end{array}
                       \right), Y=\left(
                                              \begin{array}{cc}
                                                S & 0 \\
                                                0 & T^* \\
                                              \end{array}
                                            \right), Z_1=\left(
    \begin{array}{cc}
      U^* & 0 \\
      0 & V \\
    \end{array}
  \right)\ \mbox{and}\ Z_2=\left(
                 \begin{array}{cc}
                   V^* & 0 \\
                   0 & U \\
                 \end{array}
               \right).$$
Using equations $U^*T=|T|$ and $SV^*=|S^*|$, we get
\begin{equation*}\label{equ:expression XY-1}XY=\left(
               \begin{array}{cc}
                 TS & 0 \\
                 0 & S^*T^* \\
               \end{array}
             \right)\ \mbox{and}\
Z_1 XY Z_2=\left(
                         \begin{array}{cc}
                           |T|\,|S^*| & 0 \\
                           0 & |S^*|\,|T|\\
                         \end{array}
                       \right). \end{equation*}
It follows from Lemma~\ref{thm:existence condition for polar decomposition} that
$TS$ has the polar decomposition $\Longleftrightarrow$ $\overline{\mathcal{R}(XY)}$ is orthogonally complemented, and
$|T|\,|S^*|$ has the polar decomposition $\Longleftrightarrow$ $\overline{\mathcal{R}(Z_1XYZ_2)}$ is orthogonally complemented.
It is easy to verify that $\overline{\mathcal{R}(XY)}$ is orthogonally complemented $\Longleftrightarrow$ $\overline{\mathcal{R}(Z_1XYZ_2)}$ is orthogonally complemented.
}\end{remark}

\begin{lemma}\label{lem:A and B positive and commutative}  Let $A,B\in\mathcal{L}(H)$ be such that both $\overline{\mathcal{R}(A)}$ and
$\overline{\mathcal{R}(B)}$ are orthogonally complemented. Then the following statements are valid:
\begin{enumerate}
\item[{\rm (i)}] If $A, B\in\mathcal{L}(H)_+$ and $AB=BA$, then $AB=P_{\overline{\mathcal{R}(A)}}P_{\overline{\mathcal{R}(B)}}|AB|$ is the polar decomposition;
\item[{\rm (ii)}] If $A, B\in\mathcal{L}(H)_{sa}$  and there exist two projections $p,q\in\mathcal{L}(H)$ such that $AB=pq|AB|$ is the polar decomposition, then $AB=BA$.
\end{enumerate}
\end{lemma}

\begin{proof} (i) For simplicity, we put $p=P_{\overline{\mathcal{R}(A)}}$ and $q=P_{\overline{\mathcal{R}(B)}}$. Then $AB\ge 0$ by Lemma~\ref{lem:commutative property extended-1}~(ii) and thus $|AB|=AB$.
Furthermore,  by Lemma~\ref{lem:commutative property extended-3}~(ii) we have
$$[A,q]=[p,B]=[p,q]=0,$$ hence  $pq$ is a projection. By Lemma~\ref{lem:rang characterization-1} we have
\begin{eqnarray*}\overline{\mathcal{R}\big((AB)^*\big)}=\overline{\mathcal{R}(BA)}=\overline{\mathcal{R}(Bp)}=\overline{\mathcal{R}(pB)}=\mathcal{R}(pq),
\end{eqnarray*}
therefore $pq|AB|=pqAB=AB$. Thus, we conclude from \eqref{equ:two conditions of polar decomposition} that
$AB=pq|AB|$ is the polar decomposition of $AB$.

(ii) By assumption, $qpq=(pq)^*pq$ is a projection, hence  $qpq=(qpq)^2=qpqpq$.
It follows that $(qpq-qp)(qpq-qp)^*=0$, and thus
$$qp=qpq=(qpq)^*=(qp)^*=pq,$$
which means that $pq$ is a projection.  So $pq=(pq)^*(pq)=P_{\overline{\mathcal{R}(|AB|)}}$, hence$$AB=pq|AB|=|AB|\ge 0,$$ therefore $AB=(AB)^*=B^*A^*=BA$.
\end{proof}

\begin{remark}\label{rem:improvement-2} The special case of Lemma~\ref{lem:A and B positive and commutative} (ii) above was considered in \cite[Lemma~2.4]{Ito-Yamazaki-Yanagida}, where $H$ is a Hilbert space, $A,B\in\mathcal{L}(H)_+$ are arbitrary, whereas  $p,q$ are restricted to be $P_{\overline{\mathcal{R}(A)}}$ and $q=P_{\overline{\mathcal{R}(B)}}$, respectively.
\end{remark}

We state the main result of this section as follows:
\begin{theorem}\label{thm:Ito-Yamazaki-Yanagida's main result}  Let  $T=U|T|$ and $S=V|S|$ be the polar decompositions of $T,S\in\mathcal{L}(H)$, respectively. Then the following statements are equivalent:
\begin{enumerate}
\item[{\rm (i)}] $|T|\, |S^*|=|S^*|\, |T|$;
\item[{\rm (ii)}] $TS=UV|TS|$ is the polar decomposition;
\item[{\rm (iii)}] $TS=UV|TS|$.
\end{enumerate}
\end{theorem}

\begin{proof}``(i)$\Longrightarrow$(ii)": By Lemma~\ref{lem:A and B positive and commutative}~(i),  $|T|\,|S^*|$ has the polar decomposition as
$$|T|\, |S^*|=(U^*U VV^*)\big| |T|\,|S^*|\big|. $$
Therefore, by \eqref{equ:the polar decomposition of TS} we know that
$$TS=U(U^*UVV^*)V|TS|=UV|TS|$$
is the polar decomposition of $TS$.

``(ii)$\Longrightarrow$(iii)" is obvious.

``(iii)$\Longrightarrow$(i)":  According to \eqref{equ:formula through |TS|}, we have
\begin{equation*}  U|T|\cdot |S^*|V=TS=UV|TS|=U \big||T|\, |S^*|\big|\cdot V,
\end{equation*}
hence
\begin{equation}\label{equ:preparation polar decomposition of product of operator-1}  UYV=U|Y|V,  \ \mbox{where $Y=|T|\, |S^*|$}.
\end{equation}
Note that  $U^*U|T|=|T|$, $|S^*|VV^*=|S^*|$ and $|Y|VV^*=|Y|$, so from \eqref{equ:preparation polar decomposition of product of operator-1} we can obtain
\begin{equation}\label{equ:deduced expression of Y}Y=U^*U Y VV^*=U^*U |Y| VV^*=U^*U |Y|,
\end{equation}
which gives
$|Y|^2=|Y|U^*U |Y|$ and thus $|Y|(I-U^*U)|Y|=0$, therefore $(I-U^*U)|Y|=0$. This together with \eqref{equ:deduced expression of Y}
yields $Y=|Y|\ge 0$, hence $|T|\,|S^*|=|S^*|\,|T|$ by Lemma~\ref{lem:commutative property extended-1}~(ii).
\end{proof}

\begin{remark}\label{rem:improvement-3} The equivalence of Theorem~\ref{thm:Ito-Yamazaki-Yanagida's main result} (i) and (ii) was originally stated in \cite[Theorem~2.3]{Ito-Yamazaki-Yanagida} for Hilbert space operators, whereas item (iii) of this theorem is newly added and is applied to the proof of
Theorem~\ref{thm:equivalent conditions of binormal} below.
\end{remark}

\section{\textbf{$n$-centered operators, centered operators  and binormal operators}}\label{sec:binormal-Aluthge-n centered}
In this section, we  study $n$-centered operators, centered operators and binormal operators on Hilbert $C^*$-modules, respectively. First, we introduce the term of the $n$-centered operator as follows:

\begin{definition}\label{defn:definition of n-centered operator}  Let $T\in\mathcal{L}(H)$ have the polar decomposition $T=U|T|$. Then $T$ is said to be $n$-centered for some $n\in\mathbb{N}$ if
$$\mbox{$T^k=U^k|T^k|$ is the polar decomposition for  $k=1,2,\cdots, n$.}$$
\end{definition}

To study $n$-centered operators, we need a lemma as follows:
\begin{lemma}\label{lem:Ito's 1st technique lemma}{\rm \cite[Lemma~4.2]{Liu-Luo-Xu}}\ Let $T\in\mathcal{L}(H)$ have the polar decomposition $T=U|T|$ and let $n\in\mathbb{N}$ be such that
\begin{equation}\label{equ:key commutative  assumption-00}\big[U^k|T|(U^k)^*, |T^l|\big]=0\quad  \mbox{for any $k,l\in\mathbb{N}$ with $k+l\le n+1$}.
\end{equation}
Then the following statements are equivalent:
\begin{enumerate}
\item[{\rm (i)}] $\big[U^s|T|(U^s)^*, |T^t|\big]=0$ for some $s,t\in\mathbb{N}$ with $s+t=n+2$;
\item[{\rm (ii)}] $\big[U^s|T|(U^s)^*, |T^t|\big]=0$ for any $s,t\in\mathbb{N}$ with $s+t=n+2$.
\end{enumerate}
\end{lemma}

A characterization of $n$-centered operators is as follows:
\begin{theorem}\label{thm:no1 result}  Let $T\in\mathcal{L}(H)$  have the polar decomposition $T=U|T|$. Then for each $n\in\mathbb{N}$, the following statements are equivalent:
\begin{enumerate}
\item[{\rm (i)}] $T$ is $(n+1)$-centered;
\item[{\rm (ii)}] $\big[U^k|T|(U^k)^*, |T|\big]=0$\,  for $1\le k\le n$.
\end{enumerate}
\end{theorem}

\begin{proof}``(i)$\Longrightarrow$(ii)": Suppose that $T$ is $(n+1)$-centered. Then by Definition~\ref{defn:definition of n-centered operator} $T^k\cdot T=U^k\cdot U|T^k\cdot T|$ is the polar decomposition of $T^k\cdot T$ for $1\leq k\leq n$. Then the equivalence of Theorem~\ref{thm:Ito-Yamazaki-Yanagida's main result} (i) and (ii) indicates that
\begin{equation}\label{equ:equivalent condition of n centered operator-1}\big[U|T|U^*, |T^k|\big]=\big[|T^*|, |T^k|\big]=0\quad \mbox{for $1\leq k\leq n$},
\end{equation}
which is equivalent to item (ii) of this theorem by repeating use Lemma~\ref{lem:Ito's 1st technique lemma} for $k=1,2,\cdots,n$.

``(ii)$\Longrightarrow$(i)": Suppose that  $\big[U^k|T|(U^k)^*, |T|\big]=0$ for $1\le k\le n$. Then \eqref{equ:equivalent condition of n centered operator-1}
is satisfied by repeating use Lemma~\ref{lem:Ito's 1st technique lemma}. Therefore, $T$ is an $n$-centered operator by Theorem~\ref{thm:Ito-Yamazaki-Yanagida's main result}.
\end{proof}

\begin{definition}\cite{Morrel}\ An element $T\in\mathcal{L}(H)$ is said to be centered if the following sequence
$$\cdots, T^3(T^3)^*, T^2(T^2)^*,TT^*,T^*T,(T^2)^*T^2,(T^3)^*T^3,\cdots$$
consists of mutually commuting operators.
\end{definition}

A direct application of Theorem~\ref{thm:no1 result} and \cite[Theorem~4.6]{Liu-Luo-Xu} yields the following corollary:
\begin{corollary}\label{thm:another result for centered operator} Let $T\in\mathcal{L}(H)$ have the polar decomposition $T=U|T|$. Then the following  statements are equivalent:
\begin{enumerate}
\item[{\rm (i)}] $T$ is a centered operator;
\item[{\rm (ii)}] For each $k\in\mathbb{N}$, $T^k=U^k|T^k|$ is the polar decomposition;
\item[{\rm (iii)}]For each $n\in\mathbb{N}$, $T$ is an $n$-centered operator.
\end{enumerate}
\end{corollary}
\begin{remark} The implication (i)$\Longrightarrow$(ii) of the preceding corollary was first given by Morrel and Muhly in \cite[Theorem~I]{Morrel} for Hilbert space operators. The equivalence of (i) and (ii) was  later proved in \cite[Theorem~3.2]{Ito-Yamazaki-Yanagida}  for Hilbert space operators.
\end{remark}

Next, we introduce the notations of $P_n(T)$ and $P_n(T^*)$ as follows:
\begin{definition} Let $T\in\mathcal{L}(H)$ have the polar decomposition $T=U|T|$. For each $n\in\mathbb{N}$, let $P_n(T)$ and $P_n(T^*)$ be defined by
\begin{equation}\label{equ:defn of P n T T-sta}P_n(T)=U^n(U^n)^*\ \mbox{and}\ P_n(T^*)=(U^n)^*U^n.
\end{equation}
\end{definition}
It is mentionable that if $T^n=U^n|T^n|$ is the polar decomposition, then both $P_n(T)$ and $P_n(T^*)$ are projections. With the notations of \eqref{equ:defn of P n T T-sta}, a characterization of the partial isometry can be given as follows:

\begin{lemma}\label{lem:isometries for n to n+1-1} Let $T\in\mathcal{L}(H)$ have the polar decomposition $T=U|T|$ and let
$n\in\mathbb{N}$ be such that $U^n$ is a partial isometry.  Then the following statements are equivalent:
\begin{enumerate}
\item[{\rm (i)}]$\big[P_n(T^*),UU^*]=0$;
\item[{\rm (ii)}]$U^{n+1}$ is a partial isometry;
\item[{\rm (iii)}]$\big[P_n(T),U^*U]=0$,
\end{enumerate} where $P_n(T)$ and $P_n(T^*)$ are defined by \eqref{equ:defn of P n T T-sta}.
\end{lemma}
\begin{proof} ``(i)$\Longrightarrow$ (ii)":\  Assume that $\big[P_n(T^*),UU^*]=0$. Then
\begin{align*}\big(P_{n+1}(T^*)\big)^2&=U^*\cdot P_n(T^*) UU^*\cdot P_n(T^*)U\\
&=U^*\cdot UU^*P_n(T^*)\cdot P_n(T^*)U\\
&=U^* P_n(T^*)U=P_{n+1}(T^*),
\end{align*}
therefore $P_{n+1}(T^*)$ is a projection, hence $U^{n+1}$ is a partial isometry.

``(ii)$\Longrightarrow$ (i)":\  Assume that $P_{n+1}(T^*)$ is a projection. For simplicity, we put
$p=P_n(T^*)$ and $q=UU^*$. Then the equation $P_{n+1}(T^*)=\big(P_{n+1}(T^*)\big)^2$ turns out to be
$U^*pU=U^*pqpU$. Therefore,
$$qpq=UU^*\cdot p\cdot UU^*=U\cdot U^*pqpU\cdot U^*=qpqpq,$$
which leads to $[p,q]=0$ as illustrated in the proof of Lemma~\ref{lem:A and B positive and commutative}~(ii).
The proof of (i)$\Longleftrightarrow$(ii) is then finished.

``(ii)$\Longleftrightarrow$ (iii)":\quad Since $U^n$ is a partial isometry, we know  from Lemma~\ref{prop:basic property of partial isometry} that
$(U^*)^n=(U^n)^*$ is also a partial isometry.  Note also that the polar decomposition of $T^*$ is given by \eqref{equ:the polar decomposition of T star-pre stage}, so if we replace the pair $(T,U)$ with $(T^*, U^*)$, then
\begin{eqnarray*}U^{n+1} \ \mbox{is a partial isometry}&\Longleftrightarrow & (U^*)^{n+1}\ \mbox{is a partial isometry}\\
&\Longleftrightarrow& \big[P_n(T),U^*U]=0\  \mbox{by (i)$\Longleftrightarrow$(ii)}.\qedhere
\end{eqnarray*}
\end{proof}

In what follows of this section, we study binormal operators. To this end, a lemma is stated as follows:

\begin{lemma}\label{lem:relationship between |T| and |T-star| alpha} {\rm \cite[Lemma~3.12]{Liu-Luo-Xu}}  Let $T\in\mathcal{L}(H)$ have the polar decomposition $T=U|T|$. Then for any $\alpha>0$, the following statements are valid:
\begin{enumerate}
\item[{\rm (i)}] $U|T|^\alpha U^*=(U|T|U^*)^\alpha=|T^*|^\alpha$;
\item[{\rm (ii)}] $U|T|^\alpha=|T^*|^\alpha U$;
\item[{\rm (iii)}] $U^*|T^*|^\alpha U=(U^*|T^*|U)^\alpha=|T|^\alpha$.
\end{enumerate}
\end{lemma}

\begin{definition}\label{defn:binormal operator} \cite{Campbell-1}\ An element $T$ of $\mathcal{L}(H)$ is said to be binormal if $T^*T$ and $TT^*$ are commutative, that is, $[T^*T,TT^*]=0$.  \end{definition}

\begin{remark}\label{rem:remarks on binormality}
Binormal operators are also known as weakly centered operators in \cite{Paulsen}. It follows from  Lemma~\ref{lem:commutative property extended-3}~(i)  that
\begin{equation}\label{equ:characterization of binormal-1}\mbox{$T$ is binormal}\Longleftrightarrow
\big[|T|^\alpha, |T^*|^\beta\big]=0,\ \forall\,\alpha>0, \forall\,\beta>0.\end{equation}
Note that in Definition~\ref{defn:binormal operator}  there is no demanding on the existence of the polar decomposition,
in the case that $T$ has the polar decomposition $T=U|T|$, then
$|T^*|=U|T|U^*$, hence  Theorem~\ref{thm:no1 result} indicates that
\begin{equation}\label{equ:characterization of binormal-2}\mbox{$T$ is binormal}\Longleftrightarrow  T \ \mbox{is $2$-centered}.\end{equation}
\end{remark}

Now, suppose that $T$ is binormal and meanwhile $T$ has the polar decomposition $T=U|T|$. Then by \eqref{equ:characterization of binormal-1},
Lemma~\ref{lem:Range Closure of Ta and T}~(ii) and Lemma~\ref{prop:commutative property extended-4}, we have
\begin{equation}\label{equ:basic commmatativities derived from binormal condition}\big[|T|^{\alpha}, UU^*\big]=\big[U^*U, |T^*|^{\beta}\big]=\big[U^*U, UU^*\big]=0, \ \forall\,\alpha>0, \forall\,\beta>0.
\end{equation}
Moreover,  we can prove that
\begin{eqnarray}\label{equ:commutative related binormal-1}\big[U|T|^\alpha U^*,|T|^\beta\big]=\big[U^*|T|^{\alpha}U,|T|^{\beta}\big]=0, \ \forall\,\alpha>0, \forall\,\beta>0.
\end{eqnarray}
In fact, from Lemma~\ref{lem:relationship between |T| and |T-star| alpha}~(i) and \eqref{equ:characterization of binormal-1}, we can  get
\begin{align*}\big[U|T|^\alpha U^*,|T|^\beta\big]=\big[|T^*|^\alpha,|T|^\beta\big]=0,
\end{align*}
which gives
$\big[U|T|^\beta U^*,|T|^\alpha\big]=0$ by exchanging $\alpha$ and $\beta$.
It follows that
\begin{align*}U^*|T|^{\alpha}U\cdot |T|^{\beta}&=U^*(|T|^{\alpha}\cdot U|T|^{\beta}U^*)U=U^*(U|T|^{\beta}U^*\cdot |T|^{\alpha})U\\
&=|T|^{\beta}\cdot U^*|T|^{\alpha}U.
\end{align*}

\begin{definition}\label{rem:some results of widetilde U} Suppose that $T\in\mathcal{L}(H)$  has the polar decomposition $T=U|T|$. Let
\begin{equation}\label{equ:defn of widetilde U}\widetilde{U}=U^*U^2 \ \mbox{and}\ T_{\alpha, \beta}=|T|^{\alpha} U |T|^{\beta}\ \mbox{for any  $\alpha>0, \beta>0$}.
\end{equation}
\end{definition}
\begin{remark} Much attention has been paid to the case of $\alpha=\beta=\frac12$ in the literatures, where $T_{\frac12,\frac12}$ is known as the Aluthge transformation of $T$ \cite{Aluthge}.

Let $P_k(T)$ and $P_k(T^*)$ be defined by \eqref{equ:defn of P n T T-sta} for any $k\in\mathbb{N}$. Direct computation yields that for each $n$ in $\mathbb{N}$,
\begin{align}\label{equ:formulas for n-th power of widetilde U}&\widetilde{U}^n=U^*U^{n+1}, \big(\widetilde{U}^n\big)^*\widetilde{U}^n=P_{n+1}(T^*),
\widetilde{U}^n(\widetilde{U}^n)^*=U^*U\cdot P_n(T)\cdot U^*U.
\end{align}
Therefore,
\begin{align*}\mbox{$\widetilde{U}$ is a partial isometry}&\Longleftrightarrow P_2(T^*)\ \mbox{is a projection}\\
&\Longleftrightarrow\ \mbox{$U^2$ is a partial isometry}\\
&\Longleftrightarrow\ [UU^*, U^*U]=0\ \mbox{by Lemma~\ref{lem:isometries for n to n+1-1}}.
\end{align*}
In particular, if $T$ is binormal, then $\widetilde{U}$ is a partial isometry.
\end{remark}

We give some characterizations of the binormality as follows:

\begin{theorem}\label{thm:equivalent conditions of binormal} Suppose that $T\in\mathcal{L}(H)$ has the polar decomposition $T=U|T|$. Let $\widetilde{U}$ and $T_{\alpha, \beta}$ be defined by \eqref{equ:defn of widetilde U} for $\alpha>0$ and $\beta>0$. Then the following statements are equivalent:
\begin{enumerate}
\item[{\rm (i)}] $T^2=U^2|T^2|$ is the polar decomposition;
\item[{\rm (ii)}] $T$ is binormal;
\item[{\rm (iii)}]
$\forall\, \alpha>0, \forall\, \beta>0$, $T_{\alpha, \beta}=\widetilde{U}\,|T_{\alpha, \beta}|$ is the polar decomposition;
\item[{\rm (iv)}]
$\forall\, \alpha>0, \forall\, \beta>0$, $T_{\alpha, \beta}=\widetilde{U}\,|T_{\alpha, \beta}|$;
\item[{\rm (v)}] There exist $\alpha>0$ and $\beta>0$ such that $T_{\alpha, \beta}=\widetilde{U}\,|T_{\alpha, \beta}|$.
\end{enumerate}
In each case for any $\alpha>0$ and $\beta>0$,
\begin{align}\label{eqn:formula for square root of widetilde T-alpha-beta}|T_{\alpha, \beta}|&=U^*|T|^{\alpha} U\cdot |T|^{\beta}=|T|^{\beta}\cdot U^*|T|^{\alpha} U,\\
\label{eqn:formula for square root of widetilde T-alpha-beta-star} |T_{\alpha, \beta}^*|&=|T|^{\alpha}\cdot |T^*|^{\beta}=|T^*|^{\beta}\cdot |T|^{\alpha}.
\end{align}
\end{theorem}

\begin{proof}``(i)$\Longleftrightarrow$(ii)":\ Item (i) of this theorem is true $\Longleftrightarrow$ $T$ is $2$-centered$\Longleftrightarrow$ $T$ is binormal by \eqref{equ:characterization of binormal-2}.

``(ii)$\Longrightarrow$(iii)":\ Given any $\alpha>0$ and $\beta>0$, put
$$A=|T|^{\alpha} \ \mbox{and}\ B=U|T|^{\beta}.$$ Since $|A|=A=|T|^{\alpha}$, by \eqref{equ:two conditions of polar decomposition} and Lemma~\ref{lem:Range Closure of Ta and T} we know that
\begin{equation}\label{equ:polar decomposition of A} A=U^*U|A|\ \mbox{is the polar decomposition}.
\end{equation}
Also, $B^*B=|T|^{\beta}U^*U|T|^{\beta}=|T|^{2\beta}$, which gives $|B|=|T|^{\beta}$. Therefore
\begin{equation}\label{equ:polar decomposition of B} B=U|B|\ \mbox{is the polar decomposition}.
\end{equation}
By Lemma~\ref{lem:relationship between |T| and |T-star| alpha} (i), we have
$$\mbox{$BB^*=U|T|^{2\beta}U^*=(U|T|U^*)^{2\beta}=|T^*|^{2\beta}$},$$
and thus $|B^*|=|T^*|^{\beta}$. It follows that
$$\big[|A|, |B^*|\big]=\big[|T|^{\alpha}, |T^*|^{\beta}\big]=0\ \mbox{by \eqref{equ:characterization of binormal-1}}.$$
Therefore, by Theorem~\ref{thm:Ito-Yamazaki-Yanagida's main result} we conclude that
\begin{align*}T_{\alpha,\beta}=AB=(U^*U)U|AB|=\widetilde{U}\,|T_{\alpha, \beta}|\ \mbox{is the polar decomposition}.
\end{align*}

 ``(iii)$\Longrightarrow$(iv)$\Longrightarrow$(v)" is clear.

``(v)$\Longrightarrow$(ii)":\ Let $\alpha>0$ and $\beta>0$ be such that $T_{\alpha, \beta}=\widetilde{U}\,|T_{\alpha, \beta}|$. Then
from the proof of (ii)$\Longrightarrow$(iii) above and the implication (iii)$\Longrightarrow$(i) of Theorem~\ref{thm:Ito-Yamazaki-Yanagida's main result}, we conclude that $\big[|T|^{\alpha}, |T^*|^{\beta}\big]=0,$
which in turn leads to $\big[|T|, |T^*|\big]=0$ by Lemma~\ref{lem:commutative property extended-3}~(i) since
$|T|=\left(|T|^{\alpha}\right)^\frac{1}{\alpha}$ and $|T^*|=\left(|T^*|^{\beta}\right)^\frac{1}{\beta}$. Therefore, $T$ is binormal.
This completes the proof of the equivalence of (i)--(v).

Assume now that $T$ is binormal. Then by \eqref{equ:basic commmatativities derived from binormal condition},
\begin{equation*}|T|^\alpha U |T|^\beta=|T|^\alpha (UU^*)\cdot U |T|^\beta=(UU^*)|T|^\alpha \cdot U |T|^\beta,
\end{equation*}
which is combined with \eqref{equ:defn of widetilde U} and \eqref{equ:commutative related binormal-1} to get
\begin{equation*}|T_{\alpha, \beta}|=\left(|T|^{\beta} U^* |T|^{\alpha}\cdot UU^* |T|^{\alpha} U |T|^{\beta}\right)^\frac12=U^*|T|^{\alpha} U\cdot |T|^{\beta}=|T|^{\beta}\cdot U^*|T|^{\alpha} U.
\end{equation*}
Finally, from \eqref{equ:defn of widetilde U} and Lemma~\ref{lem:relationship between |T| and |T-star| alpha}~(i) we obtain
\begin{equation*}|T_{\alpha, \beta}^*|=\left(|T|^{\alpha} U |T|^{\beta}U^*\cdot U|T|^{\beta} U^*|T|^{\alpha}\right)^\frac12=|T|^{\alpha}\cdot U|T|^{\beta} U^*=|T|^{\alpha}\cdot |T^*|^{\beta}. \qedhere
\end{equation*}
\end{proof}

\begin{remark}\label{rem:improvement-4} Putting $\alpha=\beta=\frac12$ in Theorem~\ref{thm:equivalent conditions of binormal}, the binormality of an operator can be characterized in terms of its Aluthge transformation as did in \cite{Ito}.
\end{remark}

\section{\textbf{The Moore-Penrose inverse of $n$-centered operators}}\label{sec:M-P inverse of $n$-centered operator}

In this section, we study the Moore-Penrose inverse of $n$-centered operators in the general setting of Hilbert $C^*$-modules.

\begin{lemma}\label{lem:orthogonal} {\rm (cf.\,\cite[Theorem 3.2]{Lance} and \cite[Remark 1.1]{Xu-Sheng})}\ Let $T\in \mathcal{L}(H,K)$. Then the closedness of any one of the following sets
implies the closedness of the remaining three sets:
\begin{equation}\label{equ:four ranges} \mathcal{R}(T), \mathcal{R}(T^*),  \mathcal{R}(TT^*) \ \mbox{and}\  \mathcal{R}(T^*T).\end{equation}
If $\mathcal{R}(T)$ is closed, then $\mathcal{R}(T)=\mathcal{R}(TT^*)$ and
$\mathcal{R}(T^*)=\mathcal{R}(T^*T)$.
\end{lemma}

\begin{definition}\cite{Xu-Sheng}
The Moore-Penrose inverse $T^\dag$ of $T\in \mathcal{L}(H,K)$ (if it exists)
is the unique element $X\in \mathcal{L}(K,H)$ which satisfies
\begin{equation} \label{equ:m-p inverse} TXT=T, \ XTX=X, \ (TX)^*=TX \  \mbox{and}\ (XT)^*=XT.\end{equation}
\end{definition}

 An operator $T\in \mathcal{L}(H,K)$ is called  M-P invertible if $T^\dag$ exists, which is exactly the case that $\mathcal{R}(T)$ is closed \cite[Theorem~2.2]{Xu-Sheng}; or equivalently by Lemma~\ref{lem:orthogonal}, one of
the four ranges in \eqref{equ:four ranges} is closed. When $T$ is M-P invertible, we know from \cite[Section~1]{Xu} that
\begin{equation}\label{equ:masic property of T dag-1}(T^\dag)^*=(T^*)^\dag, (TT^*)^\dag=(T^*)^\dag T^\dag, \mathcal{R}(T^\dag)=\mathcal{R}(T^*) \ \mbox{and}\ \mathcal{N}(T^\dag)=\mathcal{N}(T^*).
\end{equation}
If furthermore $T\ge 0$, then $T^\dag\ge 0$ and $(T^\dag)^\frac12=(T^\frac12)^\dag$, since
\begin{align}\label{equ:masic property of T dag-3}\big(T^\frac12\big)^*=T^\frac12\ \mbox{and}\ T^\dag=\big(T^\frac12\cdot T^\frac12\big)^\dag=(T^\frac12)^\dag \cdot (T^\frac12)^\dag.
\end{align}

Based on observations above,  two auxiliary lemmas are provided as follows:

\begin{lemma}\label{lem:trivial prop-0}Let $T\in\mathcal{L}(H)_{sa}$ and $S\in\mathcal{L}(H)$ be such that $T$ is M-P invertible and
$TS=ST$. Then $T^\dag S=ST^\dag$.
\end{lemma}
\begin{proof}A proof can be given by using the block matrix forms of $T$ and $S$.
\end{proof}
\begin{lemma}\label{lem:trivial prop-1} Let $T\in\mathcal{L}(H,K)$ be M-P invertible. Then
\begin{equation}\label{equ:relationship between M-P inverses}|T|^\dag=|(T^\dag)^*|\ \mbox{and}\ |T^*|^\dag=|T^\dag|.\end{equation}
\end{lemma}
\begin{proof} By \eqref{equ:masic property of T dag-1}--\eqref{equ:masic property of T dag-3}, we have
\begin{eqnarray*}|T|^\dag=\big((T^*T)^\frac12\big)^\dag=\big((T^*T)^\dag\big)^\frac12=\big(T^\dag (T^*)^\dag\big)^\frac12=\big(T^\dag (T^\dag)^*\big)^\frac12=|(T^\dag)^*|.
\end{eqnarray*}
Replacing $T$ with $T^*$, we obtain
$|T^*|^\dag=|\big((T^*)^\dag\big)^*|=|\big((T^*)^*\big)^\dag|=|T^\dag|.$
\end{proof}

The polar decomposition for the Moore-Penrose inverse of Hilbert space operators can be found in \cite[Proposition~2.2]{Jabbarzadeh-Bakhshkandi}. The same is also true for adjointable operators described as follows:
\begin{lemma}\label{lem:polar decomposition of T dag} Let $T\in\mathcal{L}(H)$ be M-P invertible and have the polar decomposition $T=U|T|$. Then
$T^\dag$ has the polar decomposition $T^\dag=U^*|T^\dag|$.
\end{lemma}
\begin{proof} Put $X=U^*|T^*|^\dag$. Then the expression of $T=|T^*|U$ gives
\begin{align*}TX&=|T^*|\,|T^*|^\dag\in\mathcal{L}(H)_{sa},\ TXT=T,\\
XT&=U^*|T^*|^\dag |T^*|U\in\mathcal{L}(H)_{sa},\ XTX=X.
\end{align*}
Therefore, $T^\dag=X$ by  \eqref{equ:m-p inverse} and thus $T^\dag=U^*|T^\dag|$ by \eqref{equ:relationship between M-P inverses}.

Since $\mathcal{R}(T)$ is closed, we have
\begin{equation}\label{equ:P1 T T star when Moore-Penrose inverbile}U^*U=T^\dag T\ \mbox{and}\ UU^*=TT^\dag=(T^\dag)^*T^*,
\end{equation}
so $(U^*)^*U^*=UU^*=P_{\mathcal{R}((T^\dag)^*)}$. The conclusion then follows from \eqref{equ:two conditions of polar decomposition}.
\end{proof}

To derive the main results of this section, we need a lemma as follows:
\begin{lemma}\label{lem:property of n-centered-strong} Let $T\in\mathcal{L}(H)$ have the  polar decomposition $T=U|T|$, and let $n\in\mathbb{N}$ be such that $T$ is $(n+1)$-centered. Then
\begin{equation}\big[P_k(T), |T|\big]=[P_k(T^*), |T^*|\big]=0\ \,\mbox{for $1\leq k\leq n$},
\end{equation}where $P_k(T)$ and $P_k(T^*)$ are defined by \eqref{equ:defn of P n T T-sta}.
\end{lemma}
 \begin{proof}Given any $k$ with $1\leq k\leq n$. By Definition~\ref{defn:definition of n-centered operator}  and \eqref{equ:the polar decomposition of T star-pre stage}, both $U^{k}$ and $U^{k+1}$ are partial isometries such that
 \begin{eqnarray*} \overline{\mathcal{R}(T^{k+1})}=\mathcal{R}(U^{k+1})\ \mbox{and}\ \overline{\mathcal{R}(T^k)}=\mathcal{R}(U^k),
\end{eqnarray*}
which leads to
\begin{align*}P_{k+1}(T)U|T|\cdot T^k=P_{k+1}(T)T^{k+1}=T^{k+1}=U|T|\cdot T^k.
\end{align*}
It follows that
\begin{align*}P_{k+1}(T)U|T|\cdot U^k=U|T|\cdot U^k.
 \end{align*}
Note that $[P_k(T),U^*U]=0$ by Lemma~\ref{lem:isometries for n to n+1-1}, so the equation above gives
\begin{align*}|T|U^k&=U^*\cdot U|T|U^k=U^*\cdot P_{k+1}(T)U|T|U^k=U^*U\cdot P_k(T)U^*U|T|U^k\\
&= P_k(T)\cdot U^*U\cdot U^*U|T|U^k=P_k(T)|T|U^k.
\end{align*}
Multiplying $(U^k)^*$ from the right side, we obtain
$$|T|P_k(T)=P_k(T)|T|P_k(T)=\big(P_k(T)|T|P_k(T)\big)^*=P_k(T)|T|.$$
Therefore, $\big[P_k(T), |T|\big]=0$.
Replacing the pair $(U,T)$ with $(U^*, T^*)$, we get
$$[P_k(T^*), |T^*|\big]=0\ \mbox{for $1\leq k\leq n$}.\qedhere$$
\end{proof}

Now we provide the technical result of this section as follows:

\begin{lemma}\label{lem:no3-result} Let $T\in\mathcal{L}(H)$ be M-P invertible and have the polar decomposition $T=U|T|$. Let $n\in\mathbb{N}$ be such that $T$ is $n$-centered. Then for $1\le k\le n$, $T^k$ is M-P invertible and
\begin{equation}\label{equ: Moore-Penrose inverse of T of power k}(T^k)^\dag=(T^\dag)^k\ \mbox{for all $k=1,2,\cdots,n$}.\end{equation}
\end{lemma}
\begin{proof}We prove this lemma by induction on $n$. The case $n=1$ is obvious. Assume that Lemma~\ref{lem:no3-result} is true for all
$n$-centered operators. Now, suppose that $T$ is $(n+1)$-centered. Then clearly, $T$ is $n$-centered and thus
by the inductive hypothesis, $T^k$ is M-P invertible for $1\le k\le n$ and \eqref{equ: Moore-Penrose inverse of T of power k} is satisfied.
We need only to prove that
\begin{equation}\label{equ:two tasks}T^{n+1}\ \mbox{is M-P invertible and}\ (T^{n+1})^\dag=(T^\dag)^{n+1}.\end{equation}

Let $P_m(T)$ and $P_m(T^*)$ be defined by \eqref{equ:defn of P n T T-sta} for any $m\in\mathbb{N}$. Since $T$ is $n$-centered, by  Definition~\ref{defn:definition of n-centered operator}  we have
\begin{align*}P_k(T)=P_{\overline{\mathcal{R}(T^k)}}\ \mbox{and}\ P_k(T^*)=P_{\overline{\mathcal{R}((T^k)^*)}}\,\ \mbox{for $1\le k\le n$},
\end{align*}
which is combined with \eqref{equ: Moore-Penrose inverse of T of power k} to conclude that for $1\le k\le n$,
\begin{align}\label{equ:m-p inverses of T k-111c}&(T^\dag)^k T^k=(T^k)^\dag T^k=P_{\mathcal{R}((T^k)^*)}=P_k(T^*),\\
\label{equ:m-p inverses of T k-1}&\mathcal{R}\big((T^\dag)^k\big)=\mathcal{R}\big((T^k)^\dag\big)=\mathcal{R}\big((T^k)^*\big)=\mathcal{R}\big(P_k(T^*)\big)\ \mbox{},\\
\label{equ:m-p inverses of T k-1c} &\mathcal{R}\big[\big((T^\dag)^k\big)^*\big]=\mathcal{R}\big[\big((T^k)^\dag\big)^*\big]=\mathcal{R}(T^k)=\mathcal{R}\big(P_k(T)\big).
\end{align}
Furthermore, as $T$ is $(n+1)$-centered, by Lemma~\ref{lem:property of n-centered-strong} we have
\begin{equation}\label{equ:two commutativitys from n to n+1--}\big[P_n(T),|T|\big]=\big[P_n(T^*),|T^*|\big]=0.\end{equation}
The equations above
together with Lemma~\ref{lem:trivial prop-0} yield
\begin{equation}\label{equ:two commutativitys from n to n+1}\big[P_n(T),|T|^\dag\big]=\big[P_n(T^*),|T^*|^\dag\big]=0.\end{equation}
Then by Lemma~\ref{lem:rang characterization-1}, we have
\begin{align*}\mathcal{R}\big((T^\dag)^{n+1}\big)&\subseteq\overline{\mathcal{R}\big(T^\dag (T^\dag)^n\big)}=\overline{\mathcal{R}\big(T^\dag P_n(T^*)\big)}\ \mbox{by \eqref{equ:m-p inverses of T k-1} for $k=n$}\nonumber\\
&=\overline{\mathcal{R}\big(U^*|T^*|^\dag \cdot P_n(T^*)\big)}\ \mbox{by Lemma~\ref{lem:polar decomposition of T dag} and \eqref{equ:relationship between M-P inverses}}\nonumber\\
&=\overline{\mathcal{R}\big(U^*P_n(T^*)\cdot |T^*|^\dag \big)}\ \mbox{by \eqref{equ:two commutativitys from n to n+1}}\nonumber\\
&\subseteq \mathcal{R}\big((U^{n+1})^*\big)=\mathcal{R}\big(P_{n+1}(T^*)\big),
\end{align*}
which means that
\begin{equation}\label{equ:noname-110}P_{n+1}(T^*)(T^\dag)^{n+1}=(T^\dag)^{n+1}.\end{equation}
Replacing the pair $(U,T)$ with $(U^*,T^*)$, by Lemma~\ref{lem:polar decomposition of T dag} and \eqref{equ:relationship between M-P inverses} we obtain
\begin{equation}\label{equ:m-p inverses of T k-1cc}(T^\dag)^*=(T^*)^\dag=U|(T^*)^\dag|=U |T|^\dag.\end{equation}
It follows that
\begin{align*}\mathcal{R}\big[\big((T^\dag)^{n+1}\big)^*\big]&\subseteq \overline{\mathcal{R}\big[(T^\dag)^*((T^\dag)^n)^*\big]}=\overline{\mathcal{R}\big[(T^\dag)^*P_n(T)\big]}\ \mbox{by \eqref{equ:m-p inverses of T k-1c}}\nonumber\\
&=\overline{\mathcal{R}\big(U|T|^\dag\cdot P_n(T)\big)}\ \mbox{by \eqref{equ:m-p inverses of T k-1cc}}\nonumber\\
&=\overline{\mathcal{R}\big(U P_n(T)\cdot |T|^\dag \big)}\ \mbox{by \eqref{equ:two commutativitys from n to n+1}}\nonumber\\
&\subseteq \mathcal{R}(U^{n+1})=\mathcal{R}\big(P_{n+1}(T)\big),
\end{align*}
hence
$P_{n+1}(T)[(T^\dag)^{n+1}]^*=[(T^\dag)^{n+1}]^*.$
Taking $*$-operation, we get
\begin{equation}\label{equ:noname-111}(T^\dag)^{n+1} P_{n+1}(T)= (T^\dag)^{n+1}.\end{equation}
Furthermore, from Lemma~\ref{lem:commutative property extended-3}~(ii), \eqref{equ:P1 T T star when Moore-Penrose inverbile} and \eqref{equ:two commutativitys from n to n+1--} we can obtain
\begin{align}\label{equ:two commutativitys from n to n+1--++}\big[P_n(T),T^\dag T\big]&=\big[P_n(T),U^*U\big]=0,\\
\label{equ:two commutativitys from n to n+1--++2}\big[P_n(T^*),TT^\dag\big]&=\big[P_n(T^*),UU^*\big]=0.\end{align}

Now we are ready to prove that $T^{n+1}$ is M-P invertible. To this end, we put
\begin{equation*}X=(T^n)^\dag T^n \cdot TT^\dag.\end{equation*}
Then by \eqref{equ: Moore-Penrose inverse of T of power k}, \eqref{equ:m-p inverses of T k-111c} and  \eqref{equ:two commutativitys from n to n+1--++2},
\begin{equation*}X=(T^\dag)^n T^n\cdot TT^\dag=P_n(T^*)TT^\dag=TT^\dag P_n(T^*)=TT^\dag \cdot (T^\dag)^n T^n,\end{equation*}
hence  the application of \eqref{equ:m-p inverse} gives
\begin{align}T^{n+1}&=T^n\cdot T=T^n(T^n)^\dag T^n\cdot TT^\dag T=T^n\cdot X\cdot T\nonumber\\
\label{eqn:T n+1 has an inner inverse}&=T^n \cdot TT^\dag \cdot (T^\dag)^n T^n\cdot T=T^{n+1}\cdot (T^\dag)^{n+1}\cdot T^{n+1},
\end{align}
which means clearly that $\mathcal{R}(T^{n+1})$ is closed, hence $T^{n+1}$ is M-P invertible.

 Since $T^{n+1}$ is M-P invertible and $T$ is $(n+1)$-centered, we have \begin{equation}\label{equ:Projections wrt m-p inverses-n+1 case}(T^{n+1})^\dag T^{n+1}=P_{n+1}(T^*)\ \mbox{and}\ T^{n+1}(T^{n+1})^\dag =P_{n+1}(T).
\end{equation}
We may combine \eqref{eqn:T n+1 has an inner inverse}, \eqref{equ:Projections wrt m-p inverses-n+1 case}, \eqref{equ:noname-110} with  \eqref{equ:noname-111} to get
\begin{eqnarray*}(T^{n+1})^\dag&=&(T^{n+1})^\dag \cdot T^{n+1}\cdot (T^{n+1})^\dag \nonumber\\
&=&(T^{n+1})^\dag \cdot T^{n+1} (T^\dag)^{n+1} T^{n+1}\cdot (T^{n+1})^\dag \nonumber\\
\label{eqn:contemprory of m-p inverse of T n+1} &=&P_{n+1}(T^*)\cdot (T^\dag)^{n+1}\cdot  P_{n+1}(T)=(T^\dag)^{n+1}.
\end{eqnarray*}
This completes the proof of \eqref{equ:two tasks}.
\end{proof}

\begin{theorem}\label{thm:no3-result corollay1} Let $T\in\mathcal{L}(H)$ be M-P invertible and have the polar decomposition $T=U|T|$. Then for any $n\in\mathbb{N}$, $T$ is $n$-centered if and only if $T^\dag$ is $n$-centered.
\end{theorem}
\begin{proof} Suppose that $T$ is $n$-centered. Then for  $1\le k\le n$,  by Lemma~\ref{lem:no3-result} each $T^k$ is M-P invertible
such that $(T^k)^\dag=(T^\dag)^k$.  Moreover, since $T$ is $n$-centered, we know that for $1\le k\le n$, $T^k=U^k|T^k|$ is the polar decomposition, which gives the polar decomposition of $(T^k)^\dag$ as $(T^k)^\dag=(U^k)^*|(T^k)^\dag|$ by Lemma~\ref{lem:polar decomposition of T dag}.
Therefore, $(T^\dag)^k=(U^*)^k\big|(T^\dag)^k\big|$ for all $k=1,2,\cdots,n$, which means exactly that $T^\dag$ is $n$-centered, since the polar decomposition of $T^\dag$ is given by $T^\dag=U^*|T^\dag|$.

Conversely, if $T^\dag$ is $n$-centered, then $T=(T^\dag)^\dag$ is also $n$-centered. Therefore, the proof is complete.
\end{proof}

\begin{corollary}\label{cor:trivial-no3-result corollay2} Let $T\in\mathcal{L}(H)$ be M-P invertible and have the polar decomposition $T=U|T|$. Then the following statements are valid:
\begin{enumerate}
\item[{\rm (i)}] $T$ is centered if and only if $T^\dag$ is centered;
\item[{\rm (ii)}] $T$ is binormal if and only if $T^\dag $ is binormal.
\end{enumerate}
\end{corollary}
\begin{proof}By Corollary~\ref{thm:another result for centered operator} and Theorem~\ref{thm:no3-result corollay1}, we know that
\begin{align*}\mbox{$T$ is centered}&\Longleftrightarrow \mbox{$T$ is $n$-centered for any $n\in\mathbb{N}$}\\
 &\Longleftrightarrow\mbox{$T^\dag$ is $n$-centered for any $n\in\mathbb{N}$}\\
 &\Longleftrightarrow\mbox{$T^\dag$ is centered}.
\end{align*}
In view of \eqref{equ:characterization of binormal-2}, item (ii) of this corollary is also true.
\end{proof}

\section{\textbf{An example of $n$-centered operators}}\label{sec:examples}

In this section, we provide an example for $n$-centered operators.
Let \begin{equation*}\label{equ:defn of a}e^{i\theta}=\cos\theta+i\sin\theta, \ \mbox{where}\ \theta=\frac{\sqrt{2}}{6}\pi.\end{equation*}
Then clearly, $\cos\theta>\frac12$ and
\begin{equation}\label{equ:pow of a not zero} (e^{i\theta})^k\ne 1\ \mbox{for any $k\in\mathbb{N}$}.\end{equation}
Let $\alpha$ be chosen such that
\begin{equation}\label{equ:chosen of alpha}\alpha\in(0,\frac{\pi}{2}) \ \mbox{and}\ \sin\alpha=2\cos\theta-1\in (0,1).\end{equation}
\begin{lemma}\label{lem:13k and 33k not equal to 0}Let $V=\left(
                       \begin{array}{ccc}
                        0 & 0 & 1 \\
                       \cos\alpha & \sin\alpha & 0 \\
                      \sin\alpha & -\cos\alpha & 0 \\
                   \end{array}
                \right)$, where $\alpha$ is given by \eqref{equ:chosen of alpha}.
Denote by $V^k=\left(V^{(k)}_{ij}\right)_{1\le i,j\le 3}$ for any $k\in\mathbb{N}$. Then
$V^{(k)}_{33}=V^{(1+k)}_{13}\ne 0$ for any $k\ge 2$.
\end{lemma}
\begin{proof} The eigenvalues of $V$ are $\lambda_{1}=-1$, and
\begin{eqnarray*} \lambda_2=\frac{1+\sin\alpha}{2}+\frac{\sqrt{3-2\sin\alpha-\sin^{2}\alpha}}{2}i
=\cos\theta+i\sin\theta=e^{i\theta}, \lambda_3=\overline{\lambda_2},
\end{eqnarray*}
hence $\lambda_2\lambda_3=1$ and $\lambda_2+\lambda_3=1+\sin\alpha=2\cos\theta$.
Direct computation yields
$V=PDP^{-1}$, where $D=diag\,(-1,\lambda_2,\lambda_3)$ and
\begin{eqnarray*}P&=&\left(\begin{array}{ccc}
1 & 1 & 1 \\
\frac{\sin\alpha-1}{\cos\alpha} & \frac{\sin\alpha-\lambda_{2}^{2}}{\cos\alpha} & \frac{\sin\alpha-\lambda_{3}^{2}}{\cos\alpha} \\
-1 & \lambda_2 & \lambda_3 \\
\end{array}
\right),\\
P^{-1}&=&\left(\begin{array}{ccc}
\frac{\lambda_2+\lambda_3}{\lambda_2+\lambda_3+2} & \frac{\cos\alpha}{1-\sin\alpha}\cdot\frac{\lambda_2+\lambda_3-2}{\lambda_2+\lambda_3+2} & -\frac{\lambda_2+\lambda_3}{\lambda_2+\lambda_3+2} \\
\frac{1}{\lambda_2+\lambda_3+2} & \frac{\cos\alpha}{1-\sin\alpha}\cdot\frac{1-\lambda_3}{\lambda_2+\lambda_3+2} &  \frac{\lambda_3}{\lambda_2+\lambda_3+2}\\
\frac{1}{\lambda_2+\lambda_3+2} & \frac{\cos\alpha}{1-\sin\alpha}\cdot\frac{1-\lambda_2}{\lambda_2+\lambda_3+2}& \frac{\lambda_2}{\lambda_2+\lambda_3+2} \\
\end{array}
\right).
\end{eqnarray*}
Therefore, for any $k\in\mathbb{N}$ we have
$$V^k=PD^kP^{-1}=\left(V^{(k)}_{ij}\right)_{1\le i,j\le 3},$$
where
\begin{eqnarray*}
&&V^{(k)}_{13}=\frac{(-1)^{k+1}(\lambda_2+\lambda_3)+\lambda_2^{k-1}+\lambda_3^{k-1}}
{\lambda_2+\lambda_3+2},\\
&&V^{(k)}_{33}=\frac{(-1)^{k+2}(\lambda_2+\lambda_3)+\lambda_2^{k}+\lambda_3^{k}}
{\lambda_2+\lambda_3+2}.
\end{eqnarray*}
It follows that $V^{(k)}_{33}=V^{(1+k)}_{13}$ for all $k\in\mathbb{N}$. In what follows, we show that $V^{(k)}_{13}\ne 0$ for any $k\ge 3$.

Suppose on the contrary that $V_{13}^{(k)}=0$ for some $k\ge 3$. Then the expression of $V_{13}^{(k)}$ gives
\begin{eqnarray}\label{equ: the case of 13k}
(-1)^{k+1}(\lambda_2+\lambda_3)+\lambda_2^{k-1}+\lambda_3^{k-1}=0.
\end{eqnarray}
If $k=2m$ for some $m\ge 2$, then by \eqref{equ: the case of 13k} we have $\cos\theta=\cos (2m-1)\theta$, and thus
$$0=\cos (2m-1)\theta-\cos\theta=-2\sin m\theta \sin (m-1)\theta.$$
It follows that
$(e^{i\theta})^{2m(m-1)}=1,$
which  is  contradiction to \eqref{equ:pow of a not zero}. Similarly, if $k=2m+1$ for some $m\in\mathbb{N}$, then by \eqref{equ: the case of 13k} we can get
$$\cos (m+\frac12)\theta\cdot \cos (m-\frac12)\theta=0.$$
In the case that $\cos (m+\frac12)\theta=0$, we have
\begin{eqnarray*}&&\cos (2m+1)\theta=2\cos^2 (m+\frac12)\theta-1=-1,\\
&&\cos (4m+2)\theta=2\cos^2 (2m+1)\theta-1=1.
\end{eqnarray*}
On the other hand, if $\cos (m-\frac12)\theta=0$, then $\cos (4m-2)\theta=1$. It follows that
$(e^{i\theta})^{(4m+2)(4m-2)}=1,$
which  is also  contradiction to \eqref{equ:pow of a not zero}.
\end{proof}

\begin{theorem} For any $n\ge 2$, there exist a Hilbert space $H$ and $T\in\mathcal{L}(H)$ such that $T$ is $n$-centered, whereas $T$ is not  $(n+1)$-centered.
\end{theorem}
\begin{proof} Let $\alpha$ be given by \eqref{equ:chosen of alpha} and $\{g(m)\}_{m=1}^\infty$ be any unspecified sequence such that $0<g(m)\le 4$ for all $m\in \mathbb{N}$. For each $m\in\mathbb{N}$, we put
\begin{eqnarray*} T_m=\left(
                       \begin{array}{ccc}
                        0 & 0 & \sec\alpha\cdot g(m+1) \\
                       g(m) & \tan\alpha\cdot g(m) & 0 \\
                      \tan\alpha\cdot g(m) & -g(m) & 0 \\
                   \end{array}
                \right)\in\mathbb{C}^{3\times 3}.
\end{eqnarray*}
It is easy to verify that the polar decomposition of $T_m$ is given by $T_m=V|T_m|$, where $V$ is the unitary matrix given by Lemma~\ref{lem:13k and 33k not equal to 0}, and
\begin{eqnarray*}|T_m|=diag\big(\sec\alpha\cdot g(m),  \sec\alpha\cdot g(m), \sec\alpha\cdot g(m+1)\big).
\end{eqnarray*}
As in \cite[Example~2]{Campbell-2}, we put
\begin{eqnarray*} T=\left(
                       \begin{array}{ccccc}
                        0 & 0 & 0 & 0 & \cdots\\
                       T_1 & 0 & 0 & 0 & \cdots \\
                      0 & T_2 & 0 & 0 & \cdots \\
                      0 & 0 & T_3 & 0 & \cdots\\
                      \vdots & \vdots & \vdots & \vdots &\ddots\\
                   \end{array}
                \right)\in\mathcal{L}(H),
\end{eqnarray*}
where  $H$ is the Hilbert space derived from countable number of copies of $\mathbb{C}^3$, that is,
$$H=\big\{\xi=(\xi_{i}): \xi_{i}\in\mathbb{C}^{3}\ \mbox{for each $i\in\mathbb{N}$ such that}\  \sum\limits_{i=1}^{\infty}\|\xi_{i}\|_2^2<+\infty \big\},$$
where $\Vert\cdot\Vert_2$ stands for the $2$-norm on $\mathbb{C}^3$. It is easy to verify that
\begin{equation*} T^*T=diag(T_1^*T_1, T_2^*T_2, \cdots,T_m^*T_m, \cdots )
\end{equation*}
and thus the polar decomposition of $T$ is given by $T=U|T|$, where
\begin{eqnarray*}U&=&\left(
                       \begin{array}{ccccc}
                        0 & 0 & 0 & 0 & \cdots\\
                       V & 0 & 0 & 0 & \cdots \\
                      0 & V & 0 & 0 & \cdots \\
                      0 & 0 & V & 0 & \cdots\\
                      \vdots & \vdots & \vdots & \vdots &\ddots\\
                   \end{array}
                \right)\ \mbox{and}\ |T|=diag(|T_1|,|T_2|,\cdots).
\end{eqnarray*}
It follows that for any $k\in\mathbb{N}$,
\begin{equation*}U^k|T|(U^k)^*=diag(0_{k\times k}, V^k|T_1|(V^k)^*, V^k|T_2|(V^k)^*,\cdots),
\end{equation*}
hence
\begin{eqnarray}&&\big[U^{k}|T|(U^{k})^*,|T|\big]=0\nonumber\\
\label{eqn:temp-condition wrt k-1}&&\Longleftrightarrow  V^{k}|T_{m}|(V^{k})^*\cdot|T_{m+k}|=|T_{m+k}|\cdot V^{k}|T_{m}|(V^{k})^*, \forall\,m\in\mathbb{N}.
\end{eqnarray}
For any $l\in\mathbb{N}$, let $|T_l|=\sec\alpha\cdot g(l)\cdot I+W_l$ be the decomposition of $|T_l|$, where
\begin{eqnarray*}W_l=diag\Big(0,0,\sec\alpha\cdot \big(g(l+1)-g(l)\big)\Big).
\end{eqnarray*}
Since $V$ is unitary, we have $V^{k}(V^{k})^*=I$, hence \eqref{eqn:temp-condition wrt k-1} is equivalent to
\begin{eqnarray*}V^{k}\cdot W_m\cdot  (V^{k})^*\cdot W_{m+k}=W_{m+k} \cdot V^{k}\cdot W_m\cdot (V^{k})^*,
\end{eqnarray*}
which is furthermore equivalent to
\begin{eqnarray*}&&V^{(k)}_{13}V^{(k)}_{33}\cdot \big(g(m+1)-g(m)\big)\cdot \big(g(m+k+1)-g(m+k)\big)=0,\\
&&V^{(k)}_{23}V^{(k)}_{33}\cdot \big(g(m+1)-g(m)\big)\cdot \big(g(m+k+1)-g(m+k)\big)=0,
\end{eqnarray*}
where $V^k$ is denoted by  $V^k=\left(V^{(k)}_{ij}\right)_{1\le i,j\le 3}$. Note that $V^{(1)}_{33}=0$, $V^{(2)}_{23}=\cos\alpha\ne 0$ and
$V^{(k)}_{33}=V^{(1+k)}_{13}\ne 0$ for any $k\ge 2$. The discussion above indicates that $\big[U|T|U^*,|T|\big]=0$ and
for any $k\ge 2$,
\begin{eqnarray}\hspace{-2em}&&\big[U^{k}|T|(U^{k})^*,|T|\big]=0\nonumber\\
\hspace{-2em}\label{eqn:temp-condition wrt k-2}&&\Longleftrightarrow  \big(g(m+1)-g(m)\big)\cdot \big(g(m+k+1)-g(m+k)\big)=0, \forall\,m\in\mathbb{N}.
\end{eqnarray}

Now, let $n\in\mathbb{N}$ be given. If $n=2$,  we put
\begin{equation}g(m)=m\ \mbox{for $m=1,2,3,4$ and}\ g(m)\equiv 1 \ \mbox{for all $m\ge 5$}.
\end{equation}
Then \eqref{eqn:temp-condition wrt k-2} is false for $k=2$ and $m=1$, therefore
$$\big[U|T|U^*,|T|\big]=0\ \mbox{and}\ \big[U^2|T|(U^2)^*,|T|\big]\ne 0.$$
By Theorem~\ref{thm:no1 result} we conclude that $T$ is $2$-centered, whereas $T$ is not $3$-centered.

In the case that $n\ge 3$, we put
\begin{equation*}g(1)=1, g(2)=\cdots=g(n+1)=2\ \mbox{and}\ g(m)\equiv 1 \ \mbox{for all $m\ge n+2$}.
\end{equation*}
Then for any $k=2,\cdots, n-1$, \eqref{eqn:temp-condition wrt k-2} is satisfied for all $m\in\mathbb{N}$, whereas
\eqref{eqn:temp-condition wrt k-2} is false for $k=n$ and $m=1$. Therefore,
$$\big[U^k|T|(U^k)^*,|T|\big]=0\ \mbox{for $1\le k\le n-1$ and}\ \big[U^n|T|(U^n)^*,|T|\big]\ne 0.$$
By Theorem~\ref{thm:no1 result} we conclude that $T$ is $n$-centered, whereas $T$ is not $(n+1)$-centered.
\end{proof}


\bibliographystyle{amsplain}

\end{document}